\documentclass[reqno]{amsart}
\usepackage[foot]{amsaddr}
\usepackage[bottom=3cm, top=3cm, left=3cm, right=3cm]{geometry}
\usepackage{amsmath}
\usepackage{amsthm}
\usepackage{amssymb}
\usepackage{etoolbox}
\usepackage{mathtools}
\usepackage{cases}
\usepackage{empheq}
\usepackage[shortlabels,inline]{enumitem}
\setlist[itemize]{noitemsep,nolistsep}
\setlist[enumerate]{noitemsep,nolistsep}

\usepackage[dvipsnames]{xcolor}
\usepackage{graphicx}
\definecolor{tabblue}{rgb}{0.12156862745098039, 0.4666666666666667, 0.7058823529411765}
\definecolor{taborange}{rgb}{1.0, 0.4980392156862745, 0.054901960784313725}
\definecolor{tabgreen}{rgb}{0.17254901960784313, 0.6274509803921569, 0.17254901960784313}
\definecolor{tabred}{rgb}{0.8392156862745098, 0.15294117647058825, 0.1568627450980392}
\definecolor{tabpurple}{rgb}{0.5803921568627451, 0.403921568627451, 0.7411764705882353}
\definecolor{cblue}{HTML}{2b50aa}
\colorlet{citecolor}{tabgreen}
\colorlet{linkcolor}{cblue!90}
\colorlet{urlcolor}{tabred}

\usepackage{tikz}
\usepackage{tikz-cd}
\usepackage{wrapfig}

\usepackage{hyperref}
\usepackage{url}

\hypersetup{
  breaklinks=true,
  colorlinks=true,
  linkcolor=linkcolor,
  urlcolor=urlcolor,
  citecolor=citecolor,
  bookmarksdepth=3,
  pdftitle={Gauss-Newton drifting extends natural gradient descent},
  pdfauthor={Théo Dumont, Théo Lacombe, François-Xavier Vialard},
  }

\DeclareUnicodeCharacter{0308}{o}
  
\usepackage[backref=true, backrefstyle=three, hyperref=true, maxcitenames=3, maxbibnames=50, style=alphabetic, backend=bibtex]{biblatex}
\cslet{blx@noerroretextools}\empty  
\usepackage[noabbrev,capitalize,nameinlink]{cleveref}
\crefname{equation}{}{}
\Crefname{equation}{Eq.}{}

\usepackage{autonum}
\usepackage{thmtools}
\usepackage{algorithm}
\usepackage{algpseudocodex}
\usepackage{subcaption}

\makeatletter
\renewcommand{\tocsection}[3]{%
  \indentlabel{\@ifnotempty{#2}{\bfseries\ignorespaces#1 #2\quad}}\bfseries#3}
\renewcommand{\tocsubsection}[3]{%
  \indentlabel{\@ifnotempty{#2}{\ignorespaces#1 #2\quad}}#3}
\def\l@subsection{\@tocline{2}{0pt}{2.5pc}{5pc}{}}
\makeatother

\allowdisplaybreaks

\newcounter{counter}
\numberwithin{counter}{section}
\numberwithin{equation}{section}

\newtheorem*{theorem*}{Theorem}

\newtheorem{lemma}[counter]{Lemma}

\newtheorem{proposition}[counter]{Proposition}
\newtheorem*{proposition*}{Proposition}

\theoremstyle{definition}
\newtheorem{remark}[counter]{Remark}
\newtheorem{example}[counter]{Example}
\newtheorem{definition}[counter]{Definition}

\NewCommandCopy{\proofqedsymbol}{\qedsymbol}
\AtBeginEnvironment{proof}{\renewcommand{\qedsymbol}{\proofqedsymbol}}
\AtBeginEnvironment{example}{\pushQED{\qed}\renewcommand{\qedsymbol}{$\diamondsuit$}}
\AtEndEnvironment{example}{\popQED\endexample}
\AtBeginEnvironment{remark}{\pushQED{\qed}\renewcommand{\qedsymbol}{$\triangle$}}
\AtEndEnvironment{remark}{\popQED\endexample}

\makeatletter
\newcommand\definealphabetloop[3]{%
  \ifx\relax#3\expandafter\@gobble\else\expandafter\@firstofone\fi
  {\expandafter\providecommand\expandafter*\csname#1#3\endcsname{#2{#3}}%
   \definealphabetloop{#1}{#2}}%
}%
\newcommand\definealphabet[2]{%
  \definealphabetloop{#1}{#2}abcdefghijklmnopqrstuvwxyzABCDEFGHIJKLMNOPQRSTUVWXYZ\relax
}%
\definealphabet{b}{\mathbb}%
\definealphabet{c}{\mathcal}%
\definealphabet{f}{\mathfrak}%
\definealphabet{r}{\mathrm}%
\definealphabet{s}{\mathscr}%
\definealphabet{t}{\mathtt}%
\definealphabet{bf}{\mathbf}%
\makeatother

\renewcommand{\epsilon}{\varepsilon}
\newcommand{\eps}{\varepsilon}
\renewcommand{\rho}{\varrho}
\usepackage{scalerel}

\newcommand{\W}{\operatorname{W}}
\newcommand{\PRd}{\smash{\cP_2(\bR^d)}}

\newcommand{\id}{{\operatorname{id}}}

\newcommand{\Tan}{\operatorname{Tan}}

\newcommand{\nablaa}{\nabla_{\!{\scriptscriptstyle \W}}}

\DeclareMathOperator*{\argmin}{arg\,min}
\newcommand\dd{\mathop{}\!\mathrm d}

\newcommand{\midd}{\,|\,}

\newcommand{\rhot}{\gamma}
\newcommand{\rhos}{\rho_0}
\newcommand{\Lrhos}{\smash{L^2(\rhos)}}
\renewcommand{\sg}{\text{\texttt{sg}}}
\newcommand{\Ldrift}{\cL_{\text{drift}}}
\newcommand{\LD}{\cL_{D}}

\newcommand{\Sn}{\smash{g^{n}}}
\newcommand{\Snk}[1]{\smash{g^{n}_{#1}}}
\newcommand{\Snj}{\smash{g^{n_j}}}
\newcommand{\T}[1]{\smash{f_{#1}}}
\newcommand{\hTn}{\smash{\widehat f^n}}

\usepackage{fancyhdr}
\title{Learning Monge maps with constrained drifting models}
\author[Dumont]{Théo Dumont$^\dagger$}
\author[Lacombe]{Théo Lacombe$^\dagger$}
\author[Vialard]{François--Xavier Vialard$^\dagger$}
\address{\textnormal{$^\dagger$Laboratoire d'Informatique Gaspard Monge, Université Gustave Eiffel, CNRS, F-77454 Marne-la-Vallée, France.}}
\email{\{\href{mailto:theo.dumont@univ-eiffel.fr}{theo.dumont},\href{mailto:theo.lacombe@univ-eiffel.fr}{theo.lacombe},\href{mailto:francois-xavier.vialard@univ-eiffel.fr}{francois-xavier.vialard}\}@univ-eiffel.fr}
\definecolor{tabblue}{rgb}{0.12156862745098039, 0.4666666666666667, 0.7058823529411765}
\definecolor{taborange}{rgb}{1.0, 0.4980392156862745, 0.054901960784313725}
\definecolor{tabgreen}{rgb}{0.17254901960784313, 0.6274509803921569, 0.17254901960784313}
\definecolor{tabred}{rgb}{0.8392156862745098, 0.15294117647058825, 0.1568627450980392}
\definecolor{tabpurple}{rgb}{0.5803921568627451, 0.403921568627451, 0.7411764705882353}
\definecolor{cblue}{HTML}{2b50aa}
\colorlet{citecolor}{tabgreen}
\colorlet{linkcolor}{cblue!90}
\colorlet{urlcolor}{tabred}
\usepackage{url}
\hypersetup{
  breaklinks=true,
  colorlinks=true,
  linkcolor=linkcolor,
  urlcolor=urlcolor,
  citecolor=citecolor,
  bookmarksdepth=3,
  pdftitle={Gauss-Newton drifting extends natural gradient descent},
  pdfauthor={Théo Dumont, Théo Lacombe, François-Xavier Vialard},
  }
\usepackage[noabbrev,capitalize,nameinlink]{cleveref}
\crefname{equation}{}{}
\Crefname{equation}{Eq.}{}

\title{Gauss--Newton drifting extends natural gradient descent}

\begin{document}

\begin{abstract}
    Drifting methods have recently been introduced as a promising new paradigm for training generative models, and have attracted a lot of attention so far.
    In this work, we focus on how the optimization procedure shapes their convergence properties, shedding light on the underlying machinery that may explain their empirical success.
    More specifically, when the drift field is a Wasserstein gradient, optimizing the drifting loss via plain Euclidean gradient descent seems to forfeit any hope of exploiting convexity and therefore of obtaining any convergence guarantees.
    Changing either the drift field or the optimization procedure, however, may lead to different dynamics: we show that keeping the Wasserstein gradient as drift field but replacing gradient descent with a Gauss--Newton scheme amounts to performing a natural gradient descent in parameter space.
    This connection allows us to provide global convergence guarantees for the optimization procedure under suitable convexity assumptions.

\vspace{2mm}
\end{abstract}
\maketitle

\section{Introduction}

Generative models can be framed in the following way: one wants to sample new observations from an target probability distribution $\rhot\in\PRd$, typically only known through (training) samples $x_1,\dots,x_n$.
Given an easy-to-sample source distribution $\rhos\in\PRd$ (say for instance a Gaussian distribution), which we always assume to be \emph{absolutely continuous} in this work, a way to sample from $\rhot$ is to find some map $f:\bR^d\to\bR^d$ that pushes $\rhos$ onto $\rhot$, in the sense that $X \sim \rhos$ implies $f(X) \sim \rhot$, which we denote by $f_* \rhos = \rhot$.
Finding such a map $f$, or at least a map satisfying $f_* \rhos \simeq \rhot$ for a suitable notion of distance between probability distributions, is thus at the core of most approaches underpinning generative modeling.
\smallskip

\noindent\textbf{Diffusion models.}
An inspiring example comes from diffusion models \cite{ho2020denoising}, which rely on the following observation:
if one lets $(\rho_t)_t$ be the density of the diffusion process
\begin{equation}
    \dd X_t = - X_t\dd t + \sqrt{2} \dd B_t,\qquad X_0 \sim \rhot,
\end{equation}
where $\dd B_t$ is a Wiener process, then $(\rho_t)_t$ interpolates between $\rhot$ and $\rhos$ as $t$ goes from $0$ to $+\infty$. This implies that if one consider the \emph{backward} ODE
\begin{equation}
    \dd x(s) = \underbrace{\big[ x(s) + \nabla \log(\rho_s)(x(s)) \big]}_{\eqqcolon\, v(x(s),s)} \dd s
\end{equation}
and its \emph{flow map} $\varphi_t : x(0) \mapsto x(0) + \int_0^t v(x(s),s) \dd s$, the map $f \coloneqq \varphi_\infty$ does satisfy $f_* \rhos = \rhot$.
While perfectly sound from a theoretical viewpoint, this approach raises several practical challenges: $(i)$ the \emph{score} function $\nabla \log(\rho_t)$ is unknown and has to be estimated from samples using a parameterization $\theta\mapsto s_\theta(x,t)$ (say, using a neural network); $(ii)$ the backward ODE has to be discretized in time and cannot be initialized at $t = +\infty$ in practice, which yields numerical errors that could accumulate \cite{hurault2025score}; $(iii)$ the inference step, which consists in evaluating the map $f$, is costly: one has to evaluate the model $s_\theta$ at each step of the time-discretization of the flow.
\smallskip

\noindent\textbf{A naive approach.}
Another tempting way to find a map pushing $\rhos$ onto $\rhot$, but this time with a dramatically reduced inference time, is to directly parameterize the map $f$ with some $\phi:\Theta\to\Lrhos$, $\phi:\theta\mapsto f_\theta$, using for instance a neural network of parameter space $\Theta \subset \bR^m$. Then, finding some $f_{\smash{\hat\theta}}$ pushing $\rhos$ onto $\rhot$ is equivalent to finding a minimizer $\smash{\hat\theta}$ of the loss function
\begin{equation}
    \tag{$\LD$}
    \label{eq:loss-D}
    \LD:\Theta\to\bR,\qquad
    \LD(\theta)
    \coloneqq F(f_\theta{}_*\rhos),
\end{equation}
where $F:\rho\mapsto D(\rho\midd\rhot)$ and $D$ is any divergence on $\PRd$ (that is, $D \geq 0$ and $D(\mu\midd\nu) = 0 \Leftrightarrow \mu = \nu$). Assuming that one can find a global minimizer $\smash{\hat\theta}$ of $\LD$, it is then possible to perform \emph{one-step} inference using $f_{\smash{\hat\theta}}$.
Unfortunately, this loss has no reason to exhibit convexity properties in parameter space and, as such, it seems unlikely that a plain Euclidean gradient descent over $\theta$ will eventually provide a satisfying model in practice. For future reference, the Euclidean gradient of $\LD$ can be computed using the chain rule:
\begin{equation}
    \label{eq:gradient-loss-D}
    \nabla_\theta\LD(\theta)= \int_{\bR^d} (D_\theta f_\theta(x))^\top\nablaa F(f_\theta{}_*\rhos)\circ f_\theta(x)\dd\rhos(x),
\end{equation}
where $\nablaa F$ is the \emph{Wasserstein gradient}\footnote{The Wasserstein gradient $\nablaa F(\rho)$ of $F:\PRd\to\bR$ at some $\rho\in\PRd$, if it exists, is defined by $F((\id + h)_* \rho) = F(\rho) + \int_{\bR^d} \langle\nablaa F(\rho), h\rangle \dd \rho + o(h)$; see \cite[Chap.~10]{ambrosio2008gradient}.} of $F:\PRd\to\bR$.
\smallskip

\noindent\textbf{The drifting approach.} To overcome this issue, \cite{deng2026generative} proposes an alternative approach to find a suitable parameter $\theta$.
Instead of updating $\theta$ using the standard Euclidean increment $\delta \theta = - \nabla_\theta \LD(\theta)$, they propose to look at the levels of maps: one seeks for a variation $\delta \theta$ in parameter space such that the resulting increment $f_{\theta + \delta \theta} - f_\theta$ \emph{in the space of maps} reproduces a prescribed ``drift field'' $V_\theta: \bR^d \to \bR^d$.
To do so, they formally introduce the loss
\begin{equation}
\tag{$\Ldrift$}
\label{eq:loss-drift}
\Ldrift:\Theta\to\bR,\qquad \Ldrift( \theta)\coloneqq \smash{\frac12}\|f_\theta - \sg(f_\theta + V_{\theta} \circ f_\theta)\|^2_{\rhos},
\end{equation}
where $\sg$ is the ``stop-gradient'' operator.
\begin{remark}[The stop-gradient operator]
The operator $\sg$ should be understood in the following way. $(i)$ The loss value $\Ldrift(\theta)$ is simply given by $\|f_\theta-(f_\theta+V_{\theta} \circ f_\theta)\|^2_{\rhos}=\|V_{\theta} \circ f_\theta\|^2_{\rhos}$;
$(ii)$ However, when trying to minimize $\Ldrift$, one is only allowed to iterate on $\theta$ by considering variations of the left term and what's inside the $\sg(...)$ should be treated as a constant. From a practitioner viewpoint, using the \texttt{PyTorch} syntax, this boils down to add a \texttt{.detach()} when implementing $\Ldrift$. Hence, the Euclidean gradient of $\Ldrift$ is given by
\begin{equation}
\label{eq:gradient-loss-drift}
\nabla_\theta\Ldrift(\theta) \coloneqq -\int_{\bR^d} (D_\theta f_\theta (x))^\top V_\theta \circ f_\theta(x) \dd \rhos(x).\qedhere
\end{equation}
\end{remark}
\begin{example}[Drifting with mean-shift fields]
    \label{ex:drift}
    In the seminal paper \cite{deng2026generative}, the authors propose to use the vector field
\begin{equation}
    \label{eq:drift-field}
    V_\theta(x)\coloneqq\frac{\bE_\rhot[k(x,y)(y-x)]}{\bE_\rhot[k(x,y)]}-\frac{\bE_{\rho_\theta}[k(x,y)(y-x)]}{\bE_{\rho_\theta}[k(x,y)]},
\end{equation}
    where $k$ is some kernel function, chosen to be the Laplace kernel in their work, and where $\rho_\theta\coloneqq f_\theta{}_*\rhos$. This drift field is the sum of an attraction term that drives particles toward $\rhot$ and a repulsion term that prevents collapse.
\end{example}

\begin{example}[Drifting with Wasserstein gradient fields]
    \label{ex:wass}
    In order to drive particles toward the target distribution $\rhot$, we may also use the information provided by the divergence $D$ in $\PRd$: if $F:\rho\mapsto D(\rho\midd\rhot)$ exhibits some convexity in the Wasserstein geometry, then negative its {Wasserstein gradient} $-\nablaa F(\rho)$ gives a direction of descent toward $\rhot$, in the sense that its flow map $\varphi_t$ converges to some map pushing $\rhos$ onto $\rhot$ when $t\to\infty$, and one may therefore consider the drift field
    \begin{equation}
    V_\theta(x) \coloneqq -\nablaa F(\rho_\theta)(x),
    \end{equation}
    where $\rho_\theta \coloneqq f_\theta{}_* \rhos$.
    Convexity of $F$ holds for instance when $D$ is the \emph{Kullback--Leibler (KL) divergence} (or \emph{relative entropy}) $D(\rho \midd \rhot) \coloneqq \smash{\int_{\bR^d}} \log( \smash{\frac{\dd \rho}{\dd \rhot}}) \dd \rho$ whenever $\rhot$ is log-concave~\cite[Th.~9.4.11]{ambrosio2008gradient}. See \cite{gretton2026wasserstein} for other choices for $D$.
\end{example}
\noindent It is worth mentioning that \cref{ex:drift,ex:wass} are nearly mutually exclusive: it was shown in \cite{franz2026drifting} that except when $k$ is a Gaussian kernel---a choice which had been suggested early by \cite{weber2023score}---, the mean-shift drifting field \cref{eq:drift-field} is \emph{not a Wasserstein gradient}.
    Waiving this gradient requirement allows for a much broader range of vector fields, which may, computationally, exhibit good convergence properties, as demonstrated by \cite{deng2026generative}.
    In the rest of this paper, however, we will focus on the case where the drift \emph{is} a Wasserstein gradient, as in \cref{ex:wass}, which will allow for a useful interpretation of the optimization as a natural gradient descent and for convergence guarantees. In this case, one has the following result, which is direct by comparing \cref{eq:gradient-loss-D,eq:gradient-loss-drift}.

    \begin{proposition}[Euclidean gradient descent for the drifting loss, {\cite[Prop.~8]{gretton2026wasserstein}}]
    Assume that the drift field is $V_\theta\coloneqq-\nablaa F(f_\theta{}_*\rhos)$. Then for all $\theta\in\Theta$, $\nabla_\theta\Ldrift(\theta)=\nabla_\theta\LD(\theta)$.
    As such, the (Euclidean) gradient descent of the drifting loss \cref{eq:loss-drift} is identical to that of \cref{eq:loss-D} and reads
    \begin{equation}
    \tag{\textsc{eucl.~gd}}
    \label{eq:eucl-gd}
        \theta_{k+1}
        = \theta_k - \tau \nabla_\theta\LD(\theta).
    \end{equation}
    \end{proposition}
Yet, one is not limited to this and one may use any optimization scheme such as, for instance, the \textsc{adamw} scheme used in \cite{deng2026generative}.
In the next section, we show that when one rather uses the \emph{Gauss--Newton} scheme, the evolution in parameter space can be seen as a natural gradient descent and convergence guarantees to the minimizer can be established.


\section{Natural gradient interpretation of drifting models with Gauss--Newton}

In this section, we first establish the equivalence between the Gauss--Newton scheme for the drift loss and the natural gradient descent method (\cref{cor:natural-gf}). We then capitalize on this result to obtain a convergence result for this scheme, performed over an expressive enough class of neural networks (\cref{prop:cvg-Kn}).

\subsection{The Gauss--Newton scheme}
\begin{definition}[Gauss--Newton scheme]
Let $L:\Theta\to\bR$ be some loss function of the form $L(\theta)= \smash{\frac12}\|\ell(\theta)\|^2$, where $\ell:\Theta\to \bR^d$ is differentiable.
The \emph{Gauss--Newton scheme} aims at solving the minimization problem $\min_{\delta \theta} \|\ell(\theta + \delta \theta)\|^2$ by iteratively solving the locally linearized problem
\begin{equation}
    \label{eq:gauss-newton}
    \theta_{k+1}=\theta_k+\tau \delta\theta_k,\qquad\text{with }\delta\theta_k\in \argmin_{\delta \theta} \|\ell(\theta_k) + J_{\theta_k}\delta\theta \|^2,
\end{equation}
where $J_{\theta_k}\coloneqq D_\theta \ell(\theta_k)\in\bR^{d\times m}$ is the Jacobian matrix of $\ell$ at $\theta_k$ and $\tau > 0$ is some fixed step size.
If $\smash{J_{\theta_k}^\top J_{\theta_k}}\in\bR^{m\times m}$ is non-singular, this regression problem admits the unique solution
\begin{equation}
\label{eq:GN_update_explicit}
\delta\theta = - [J_{\theta_k}^\top J_{\theta_k}]^{-1} J_{\theta_k}^\top \ell(\theta_k).
\end{equation}
Whenever $\smash{J_{\theta_k}^\top J_{\theta_k}}$ is singular, one can replace the inverse by a Moore--Penrose pseudoinverse to obtain the minimal-norm solution to \cref{eq:gauss-newton}. Adding elements of $\ker J_{\theta_k}$ then describes the whole set of solutions.
\end{definition}
\begin{remark}[Link with Newton's method]
Newton's method amounts to preconditioning the gradient $\nabla_\theta L(\theta)=J_\theta^\top \ell(\theta)$ with the inverse of the Hessian matrix of $L$, which is the sum of $\smash{J_\theta^\top} J_\theta$ and a second-order derivative.
As it can be seen from \cref{eq:GN_update_explicit}, the Gauss--Newton scheme merely uses the inverse of $\smash{J_\theta^\top} J_\theta$ as a preconditioner, therefore approximating the full Hessian of $L$ by discarding the second-order derivatives.
\end{remark}

\subsection{Gauss--Newton for the drifting loss.}
Applying the Gauss--Newton scheme \cref{eq:gauss-newton} to the drifting loss \cref{eq:loss-drift}, we get that the successive parameter updates read $\theta_{k+1}=\theta_k+\tau\delta\theta_k$ with
\begin{subequations}
\begin{align}
    \delta\theta_k
    &\in \argmin_{\delta \theta} \| - V_{\theta_k}\circ f_{\theta_k}+D_\theta f_{\theta_k} . \delta\theta \|^2_{\rhos}
    \label{eq:gauss-newton-drift-implicit}
    \\
    &= \Big[\int_{\bR^d}(D_\theta f_{\theta_k})^\top D_\theta f_{\theta_k}\dd\rhos\Big]^{-1} \int_{\bR^d} (D_\theta f_{\theta_k})^\top V_{\theta_k}\circ f_{\theta_k} \dd\rhos,
    \label{eq:gauss-newton-drift-explicit}
\end{align}
where \cref{eq:gauss-newton-drift-explicit} holds assuming the invertibility of $G_{\theta_k}\coloneqq\smash{\int_{\bR^d}}(D_\theta f_{\theta_k})^\top D_\theta f_{\theta_k}\dd\rhos$, while \cref{eq:gauss-newton-drift-implicit} always holds and can be used faithfully in practice, moreover, without having to pay the computational cost of inverting this $m\times m$ matrix.
Replacing $V_\theta$ by its value as the Wasserstein gradient of $F$ yields the more compact
\begin{equation}
    \label{eq:gauss-newton-drift-explicit2}
\theta_{k+1}= \theta_k - \tau  \smash{G_{\theta_k}^{-1}}\nabla_\theta\LD(\theta_k).
\end{equation}
\end{subequations}
\smallskip

\noindent\textbf{Relation to natural gradient flows.}
Looking at \cref{eq:gauss-newton-drift-explicit2}, it is explicit that the only (but crucial) difference with the Euclidean gradient flow \cref{eq:eucl-gd} is the preconditioning matrix $\smash{G_\theta^{-1}}$. As noted by \cite{king2026gauss}, this hints at a connection with the \emph{natural gradient} schemes, which we detail below.

\begin{definition}[$\Lrhos$-natural gradient descent]
    Let $\theta\mapsto f_\theta$ be differentiable from $\Theta$ to $\Lrhos$ and whose differential is injective at all $\theta\in\Theta$, and let $f\mapsto F(f_*\rhos)$ be differentiable from $\Lrhos$ to $\bR$.
    Then the \emph{$\Lrhos$-natural gradient descent} of $\theta\mapsto F(f_\theta{}_*\rhos)$ on $\Theta$ is its gradient descent with respect to the pullback metric (see \cref{def:pullback} in the appendix) of the flat $\Lrhos$-metric by the map $\theta\mapsto f_\theta$.
\end{definition}
\noindent This procedure takes origins in the seminal work of \cite{amari1998natural} that pulled back the Fisher--Rao metric from $\PRd$ to $\Theta$.
Using the definition of the pullback metric, one sees that it is in our case precisely $G_\theta$ (see \cref{app:Gtheta_pullback} in the appendix for a short proof).
This yields the following proposition, another proof of which can be found in \cite[Cor.~3.8]{dumont2026learning}.

\begin{proposition}[Gauss--Newton for the drifting loss is a natural gradient descent]
    \label{cor:natural-gf}
    Let $\theta\mapsto f_\theta$ be differentiable and such that $G_\theta$ is invertible, and let $F:\PRd\to\bR$ be Wasserstein differentiable.
    Assume that the drift field is $V_\theta\coloneqq-\nablaa F(f_\theta{}_*\rhos)$.
    Then the Gauss--Newton scheme \cref{eq:gauss-newton-drift-implicit} on the drifting loss \cref{eq:loss-drift} is the $\Lrhos$-natural gradient descent of \cref{eq:loss-D} on $\Theta$.
\end{proposition}

\noindent As such, whereas the standard Euclidean gradient descent \cref{eq:eucl-gd} imposes a flat metric on the parameter space $\Theta$ and yields an evolution in a curved $\Lrhos$, the $\Lrhos$-natural gradient descent imposes a simpler geometry on $\Lrhos$: its flat Hilbert structure. The image $f_{\theta_k}=\phi(\theta_k)$ therefore draws the best approximation of the gradient descent of $\smash{\widetilde F}:\Lrhos\to\bR$ in the subset $\phi(\Theta)$, akin to the Dirac--Frenkel principle \cite[Sec.~II.1]{lubich2008quantum}.
Whenever the functional is well-behaved in $\PRd$ with the Wasserstein metric (such as the KL divergence, convex whenever the reference measure is log-concave) and since this space is strongly linked\footnote{The pushforward mapping $f\mapsto f_*\rhos$ can be seen as an informal Riemannian submersion between $\Lrhos$ and the Wasserstein space \cite{otto2001geometry}.} to $\Lrhos$, this pullback geometry on $\Theta$ seems better suited for guaranteeing convergence of the descent taking place on $\Lrhos$, as we demonstrate in the next paragraph.
As a side note, \cref{cor:natural-gf} also shows that in the case where the drift field $V_\theta$ is not necessarily a Wasserstein gradient, the Gauss--Newton scheme for the drifting loss can be seen as an \emph{extension} of the natural gradient descent method.

\begin{remark}[On the degeneracy of $G_\theta$]
In practice, the Jacobian $D_\theta f_\theta$ is not injective, as there are directions $\delta\theta$ that do not change the model $f_\theta$. This implies that $G_\theta$ is degenerate, and therefore that minimizers of \cref{eq:gauss-newton-drift-implicit} are not unique. Yet, one may ask for the one of \emph{minimum norm}: it is given by \cref{eq:gauss-newton-drift-explicit2} but where the inverse $\smash{{G_\theta^{-1}}}$ is replaced by the Moore--Penrose pseudoinverse $\smash{G_\theta^\dagger}$.
Alternatively, one may also regularize \Cref{eq:gauss-newton-drift-implicit} with a weighted $\ell^2$ norm on $\theta$, which amounts to the Levenberg--Marquardt method, that is, replacing the inverse $\smash{{G_\theta^{-1}}}$ in \cref{eq:gauss-newton-drift-explicit2} by $\smash{(G_\theta+\lambda I)^{-1}}$.
\end{remark}
\smallskip

\noindent\textbf{Convergence guarantees for the Gauss--Newton drifting.}
We now present our main result, stated in \cref{prop:cvg-Kn} below.
First, note that the scheme \cref{eq:gauss-newton-drift-explicit2} is an \emph{explicit} gradient descent. As such, one cannot hope for convergence results on it without assuming some smoothness on the functional $F$---typically, some Lipschitz continuity on its gradient, an assumption which is rarely satisfied. To get those convergence results, one may rather consider the \emph{implicit} version of this scheme.
This gradient descent, written on the set $K\subset\Lrhos$ of functions that can be described by the neural network class\footnote{For instance, $K$ can be the set of convex functions in the case of ICNNs~\cite{amos2017input}, LSE networks~\cite{calafiore2019log} or Max-Affine models \cite{ghosh2021max}, the set of gradients of convex functions in the context of learning OT maps~\cite{chaudhari2023learning}, nonnegative functions...}, reads $\T{0}\in K$ and for all $k\geq0$,
\begin{equation}
    \label{eq:jko}
    \T{k+1}=\argmin_{f\in K} \Phi(f,\T{k}),
\end{equation}
where $\Phi:\Lrhos\times\Lrhos\to\bR$ is the proximal mapping
\begin{equation}
    \Phi(f,\overline f)\coloneqq \smash{\widetilde F}(f)+\frac1{2\tau}\|f-\overline f\|^2_{\rhos}.
\end{equation}
with $\smash{\widetilde F}:f\mapsto F(f_*\rhos)$.
If $K$ is closed and convex and $\smash{\widetilde F}$ is convex on $K$, then $\T{k}\to\argmin_{K} \smash{\widetilde F}$ with exponential convergence rate.
For instance, this is true for the KL divergence with respect to some $\lambda$-log-concave measure $\rhot$ with $\lambda>0$ and $K$ the set of gradients of convex functions~\cite[Prop.~3.1]{dumont2026learning}.
In practice, however, one merely has access to parameterizations $\phi:\theta\mapsto f_\theta$, which take values in some subset of $\phi(\Theta)\subset K$. Yet, in the large parameter limit, the neural network class may become expressive enough to approximate the whole set $K$, as in the case of, for instance, ICNNs~\cite{chen2018optimal}.
In that case, there exists some increasing sequence $K_n\coloneqq \phi_n(\Theta_n)$ that grows to $K$ and one can hope that the scheme \cref{eq:jko} on those $K_n$ also converges to the minimizer of $F$ in the limit $n\to\infty$. This scheme consists in the iterates $(\Snk{k})_k$ defined by $\Snk{0}\in K_n$ and for all $k\geq0$,
\begin{equation}
    \label{eq:jko_Kn_exact}
    \Snk{k+1}\in\argmin_{g\in K_n} \Phi(g,\Snk{k}).
\end{equation}
Since it is not given that such a minimizer exists in $K_n$ (as, while we assume $K$ closed and convex, no such assumption can be made on $K_n$ in general) and because, in practice, a minimizer is often not reached exactly, it is actually more relevant to relax this exact optimality and to rather solve
    \begin{equation}
        \label{eq:jko_Kn}
        \Snk{k+1}\in K_n
        \quad\text{such that}\quad
        \Phi(\Snk{k+1},\Snk{k})
        \leq
        \inf_{g\in K_n}\{\Phi(g,\Snk{k})\}+\eps_n,
    \end{equation}
    where $\eps_n>0$ is a tolerance satisfying $\eps_n\to0$. See \cref{rem:eps} for further discussions on this relaxation. Such a $\Snk{k+1}$ exists as soon as the infimum in \cref{eq:jko_Kn} is finite, which is the case if for instance $F$ is bounded below. Since all the statements below are asymptotic in $n$, we assume without loss of generality that $\Snk{k}$ is well defined for all $n,k\geq0$.
    We now state our main result, which shows that as the size of the neural network parameterization $K_n$ of $K$ increases, long-time iterates of \cref{eq:jko_Kn} converges to the minimizer of $\smash{\widetilde F}$ on $K$.
\begin{proposition}[Convergence of the discretized scheme to the minimizer]
    \label{prop:cvg-Kn}
Let $K\subset\Lrhos$ be a closed convex set and let $\smash{\widetilde F}:\Lrhos\to\bR$ be strongly continuous, weakly lower semicontinuous, convex and lower-bounded on $K$.
Let $\phi_n:\Theta_n\to\Lrhos$ be neural networks of (nonempty) finite-dimensional parameter spaces $\Theta_n$ and let $K_n\coloneqq \phi_n(\Theta_n)\subset K$. Assume that $K_n\subset K_{n+1}$ for all $n\geq0$ and that $\smash{\overline {\bigcup_n K_n}=K}$. Let $\tau>0$ be some fixed step size and let $(\eps_n)_n$ be positive tolerances with $\eps_n\to0$. Assume that the initializations satisfy $\Snk{0}\to\T{0}$. Then, the iterates $\Snk{k}$ of \cref{eq:jko_Kn} satisfy
\begin{equation}
\smash{\lim_{k\to\infty}\lim_{n\to\infty}}\Snk{k}=f^\star,\qquad \text{where }f^\star=\smash{\argmin_{K}}\,\smash{\widetilde F},
\end{equation}
and therefore there exists a sequence $n_k\to\infty$ such that $g_k^{n_k}\to f^\star$.
\end{proposition}
\noindent The result directly follows from the following lemma, which states that when the size of $K_n$ increases, one asymptotically recovers the iterates of the implicit gradient descent \cref{eq:jko} of the \emph{convex} functional $\smash{\widetilde F}$ on the closed convex set $K$, for which global convergence to $f^\star$ as $k\to\infty$ is guaranteed~\cite[Prop.~3.1]{dumont2026learning}.
\begin{lemma}[Convergence of the discretized scheme to the continuous scheme]
    \label{lem:cvg-Kn}
    Let the assumptions of \cref{prop:cvg-Kn} be satisfied.
    Then for all $k\geq0$, the iterate $\Snk{k}$ of \cref{eq:jko_Kn} converges in $\Lrhos$ to the solution $\T{k}$ of \cref{eq:jko} as $n\to\infty$.
\end{lemma}
\begin{proof}
    Let us proceed by induction. By definition, $\Snk{0}=\id=\T{0}$. Assume now that for some $k\geq0$, $\Snk{k}\to \T{k}$ and let us show that $\Snk{k+1}\to \T{k+1}$.
    For the sake of conciseness, let us write $\Sn\coloneqq \Snk{k+1}$ and $f\coloneqq \T{k+1}$.
    Let $J_{n,k}\coloneqq \Phi(\cdot,\Snk{k})$ be the functional in \cref{eq:jko_Kn} and let $J_k\coloneqq \Phi(\cdot,\T{k})$ be that in \cref{eq:jko}.
    By $1/\tau$-strong convexity of $J_k$ and optimality of $f$ for $J_k$ on the convex set $K$, one gets
    \begin{equation}
        \label{eq:strong-cvx-and-optimality}
        \frac1{2\tau}\|\Sn-f\|^2_{\rhos}\leq J_k(\Sn)-J_k(f).
    \end{equation}
    It is therefore enough to show that $J_k(\Sn)$ converges to $J_k(f)$ to obtain the desired result.
    For that, we will compare it to $J_{n,k}(\Sn)$. Expanding the squared norms in $J_k$ and $J_{n,k}$ yields the following:
    \begin{equation}
        \label{eq:Jkdiff}
        J_k(\Sn)-J_{n,k}(\Sn)
        =\frac1{2\tau}\big(\|\T{k}\|^2_{\rhos}-\|\Snk{k}\|^2_{\rhos}+2\langle \Sn,\Snk{k}-\T{k}\rangle_{\rhos}\big).
    \end{equation}
    Now, consider some mappings $\hTn\in K_n$ such that $\hTn\to f$ (see \cref{lem:Tn} for the proof of their existence).
    Since $F$ is strongly continuous, $\smash{\widetilde F}(\hTn)\to \smash{\widetilde F}(f)$; and since by induction hypothesis $\Snk{k}\to \T{k}$, one has $\|\hTn-\Snk{k}\|^2_{\rhos}\to \|f-\T{k}\|^2_{\rhos}$. Finally, $J_{n,k}(\hTn)\to J_k(f)$.
    By $\eps_n$-optimality of $\Sn$ for $J_{n,k}$ on $K_n$ and since $\hTn\in K_n$, one also has $J_{n,k}(\Sn)\leq J_{n,k}(\hTn)+\eps_n$, and therefore, using that $\eps_n\to0$,
    \begin{equation}
        \label{eq:limsupinf}
        \limsup_{n\to\infty} J_{n,k}(\Sn)
        \leq\limsup_{n\to\infty} J_{n,k}(\hTn)
        = J_k(f)<\infty,
    \end{equation}
    and thus $J_{n,k}(\Sn)$ is bounded above by some $M\in\bR$. Then $\|\Sn-\Snk{k}\|^2_{\rhos}\leq 2\tau(M-m)$, where $m\in\bR$ is some lower bound on $\smash{\widetilde F}$. Since by induction hypothesis $(\Snk{k})_n$ converges to $\T{k}$, it is bounded, and $(\Sn)_n$ is therefore bounded as well.
    Hence, using once again that $\Snk{k}\to\T{k}$ and the Cauchy--Schwarz inequality, the whole expression \cref{eq:Jkdiff} tends to $0$ as $n\to\infty$. This in turns implies that
    \begin{equation}
    \label{eq:limsup}
    \limsup_{n\to\infty} J_{n,k}(\Sn)=\limsup_{n\to\infty} J_k(\Sn).
    \end{equation}
    By definition of the limit inferior, there exists some $(n_j)_j$ such that $J_k(\Snj)\to \liminf_n J_k(\Sn)$ as $j\to\infty$. Since $(\Sn)_n$ is bounded, $(\Snj)_j$ is bounded as well and one can further extract a subsequence (which we still index with $n_j$) that converges weakly to some $g^\infty$. Since $K$ is convex and strongly closed in the Hilbert space $\Lrhos$, it is weakly closed \cite[Th.~3.7]{brezis2011functional}, and $g^\infty\in K$.
    By the weak lower semicontinuity of $\smash{\widetilde F}$ and of the norm, hence of $J_k$, one gets
    \begin{equation}
        \label{eq:final-ineq}
        J_k(g^\infty)
        \leq
        \liminf_{j\to\infty} J_k(\Snj)
        =
        \liminf_{n\to\infty} J_k(\Sn)
        \leq
        \limsup_{n\to\infty} J_k(\Sn)
        \!\stackrel {\cref{eq:limsup}}{{\!\vphantom \leq} =}\!
        \limsup_{n\to\infty} J_{n,k}(\Sn)
        \!\stackrel {\cref{eq:limsupinf}}\leq \!
        J_k(f).
    \end{equation}
    The optimality of $f$ for $J_k$ on $K$ yields that $g^\infty=f$ and that all the inequalities in \cref{eq:final-ineq} are equalities. In particular, $J_k(\Sn)\to J_k(f)$ and re-injecting in \cref{eq:strong-cvx-and-optimality} yields that $\Sn\to f$.
    \end{proof}
\noindent It is worth mentioning that without any uniform bound in $k$ on the difference between $\Snk{k}$ and $\T{k}$, there is no hope of obtaining asymptotic convergence results for the iterates \cref{eq:jko_Kn} at fixed $n$.
\begin{remark}[On the relaxation from \cref{eq:jko_Kn_exact} to \cref{eq:jko_Kn}]
    \label{rem:eps}
    Relaxing the exact minimization \cref{eq:jko_Kn_exact} over $K_n$ into the $\eps_n$-optimality condition \cref{eq:jko_Kn} is closer to what an optimizer actually returns. It is also needed in general: when $K_n=\phi_n(\Theta_n)$ is the image of a finite-dimensional parameter space under a nonlinear map, $K_n$ is in general not closed (nor convex) in $\Lrhos$ \cite{petersen2020topologicalpropertiessetfunctions}. An exact minimizer of $\Phi(\cdot,\Snk{k})$ over $K_n$ therefore need not exist unless one asks, for instance, for compactness of the parameter space $\Theta_n$ and continuity of $\phi_n$. In that case, one may take $\eps_n=0$ in the statement of \cref{prop:cvg-Kn} and its proof.
\end{remark}
\begin{remark}[The case of the KL divergence on the set of gradients of convex functions]
    Let $F$ be the KL divergence with respect to some $\lambda$-log-concave measure $\gamma$, with $\lambda>0$. Then $F$ is minimized at $\gamma$, l.s.c.~with respect to the Wasserstein topology~\cite[Cor.~15.7]{ambrosio2021lectures} and $\lambda$-convex along generalized geodesics \cite[Th.~9.4.11]{ambrosio2008gradient} on $\PRd$. If $K$ is a closed convex subset of the set of gradients of convex functions, then $\smash{\widetilde F}$ is therefore convex and weakly l.s.c.~on $K$~\cite[Lem.~B.3, B.4, B.5]{dumont2026learning}.
    The strong continuity assumption on $\smash{\widetilde F}$, however, is not satisfied by the KL divergence. Yet, this assumption can actually be weakened: one merely needs the \emph{existence} of a recovery sequence $\hTn\in K_n$ such that $\hTn\to f$ and $\smash{\widetilde F}(\hTn)\to \smash{\widetilde F}(f)$. One can for instance find such a sequence in the case of $C^1$-diffeomorphisms on a compact domain $\Omega$ whenever $\rho_0$ and $\gamma$ have positive and continuous densities $\mu_0$ and $\nu$. The KL divergence reads in this case
    \begin{equation}
        \smash{\widetilde F}(f)=\int_\Omega \mu_0(x)\log\frac{\mu_0(x)}{\nu(f(x))\det Df(x)}\dd x.
    \end{equation}
    Then, given that one may approximate $f$  by some $\hTn$ in the $C^1$ topology---which is true for specific neural network classes~\cite{hornik1990universal}, one gets $\smash{\widetilde F}(\hTn)\to \smash{\widetilde F}(f)$ and the recovery sequence needed in the proof of \cref{lem:cvg-Kn}, allowing to apply \cref{prop:cvg-Kn} to obtain the convergence of the discretized scheme \cref{eq:jko_Kn} to the minimizer of $\smash{\widetilde F}$ on $K$.
\end{remark}

\section{Conclusion}
\noindent\textbf{Contribution.}
    In this work, we study the convergence properties of the optimization of the drifting loss when the drift field is the Wasserstein gradient of some function $F$. Under suitable convexity assumptions on $F$, we show that, unlike with a plain Euclidean gradient descent, one can guarantee convergence to the minimizer of $F$ if using a Gauss--Newton scheme to minimize the drift loss. This is done by exploiting the fact that the Gauss--Newton scheme is equivalent to performing a natural gradient descent on parameter space, which, as a side note, allows to see the Gauss--Newton optimization of the drifting loss with an arbitrary drift field as a generalization of natural gradient descent for non-gradient vector fields.
\smallskip

\noindent\textbf{Open questions.}
\emph{A quest for well-behaved functionals on $\Lrhos$.}
In order to conclude to the convergence of the Gauss--Newton scheme, \cref{prop:cvg-Kn} requires the convexity of the functional $\smash{\widetilde F}:f\mapsto F(f_*\rhos)$ on the space $\Lrhos$ of maps. Although \cite[Lem.~B.4]{dumont2026learning} establishes some conditions on $F$---namely, convexity along generalized convexity---that ensures this \emph{lifted} convexity, few functionals on $\PRd$ do satisfy them, and if they do, lifted convexity is only guaranteed on the set of gradients of convex functions (that is, optimal transport maps). It would therefore be of great interest to find other examples of functionals $F$ on $\PRd$ that ensures convexity of $\smash{\widetilde F}$, or a mere Polyak--Łojasiewicz condition~\cite{polyak1964gradient,lojasiewicz1963topological} on it, or even functionals with no apparent convexity but with convergence properties, such as the Maximum Mean Discrepancy~\cite{chizat2026quantitative}.
\\\emph{Beyond gradient drift fields.}
While our analysis focuses on Wasserstein gradients, the choice of drift fields is not limited to this setting~\cite{deng2026generative}. It is worth investigating whether one may find examples or classes of drift fields for which it is possible to establish convergence guarantees for the optimization of the drifting loss as well. A particularly interesting direction may be to look at vector fields whose flow map converges to a map pushing $\rhos$ onto $\rhot$, despite lacking an apparent gradient or convexity structure. Understanding which specific properties of the drift field are sufficient to ensure convergence in this general setting could substantially widen the scope of the drifting method.
\\\emph{Beyond Gauss--Newton.}
We demonstrated that switching from plain Euclidean gradient descent to a Gauss--Newton scheme leads to different dynamics with, in our setting, improved convergence guarantees.
This naturally raises the question of how other optimizers interact with the drifting loss, such as, for instance, the \textsc{adamw} optimizer considered in \cite{deng2026generative}.
More generally, we view the drifting method as a pair $($drift field, optimizer$)$, with both components offering degrees of freedom that can be exploited to shape the theoretical and computational behavior of the optimization. As such, our choice of gradient drift field and Gauss--Newton optimizer provides one example of a more general framework. Exploring both components of the pair may allow one to tailor the drifting method to their needs, be it to obtain convergence guarantees or computational efficiency.

\section*{Acknowledgements}
This research is partly supported by the Bézout Labex, funded by ANR, reference ANR-10-LABX-58.
TL is supported by the ANR project TheATRE, reference ANR-24-CE23-7711.

\addtocontents{toc}{\protect\setcounter{tocdepth}{1}}
\renewcommand*{\bibfont}{\small}
\printbibliography
\newpage

\appendix
\phantomsection
\addcontentsline{toc}{section}{\textbf{Appendix}}

\addtocontents{toc}{\protect\setcounter{tocdepth}{0}}

\section{Notation}

\noindent
Let $d\geq1$.
    We denote by $\PRd$ the set of probability measures on $\bR^d$ {with finite second-order moment}.
    If some $\rho\in\PRd$ is {absolutely continuous} with respect to some $\gamma\in\PRd$, we denote by $\smash{\frac{\dd \rho}{\dd \gamma}}$ the corresponding Radon--Nikodym derivative.
    The {pushforward} $f_*\rho$ of some $\rho\in\PRd$ by some measurable map $f:\bR^d\to\bR^d$ is the probability measure defined on Borel sets $A$ by $f_*\rho(A)\coloneqq \rho(f^{-1}(A))$.
    In this work, $\rhos\in\PRd$ is an absolutely continuous probability measure, and $\Lrhos$ is the Hilbert space of measurable functions $f:\bR^d\to\bR^d$ that are squared-integrable with respect to $\rhos$, endowed with its norm $\smash{\|\cdot\|_{\rhos}}$ and scalar product $\smash{\langle\cdot,\cdot\rangle_{\rhos}}$.

\section{Additional definitions and omitted proofs}

\begin{definition}[Pullback metric]
    \label{def:pullback}
    Let $\Theta$ be a finite-dimensional manifold and $M$ be a (possibly infinite-dimensional with a strong metric~\cite{schmeding2022introduction}) Riemannian manifold with metric $g$. Let $\phi:\Theta\to M$, $\theta\mapsto f_\theta$ be differentiable and whose differential is injective al all $\theta\in\Theta$.
    Then the \emph{pullback metric} of $g$ by $\phi$, which we note $\phi^*g$, is the metric on $\Theta$ defined as
    \begin{equation}
    (\phi^*g)_\theta(\delta\theta,\delta\theta)\coloneqq g_{f_\theta}(d_\theta \phi[\delta\theta],d_\theta \phi[\delta\theta])
    \qquad
    \text{for any $\theta\in\Theta$ and $\delta\theta\in \Tan_\theta\Theta$.}
    \end{equation}
\end{definition}
The following proposition shows that in our case, the pullback of the flat metric on $\Lrhos$ by $\phi$ enjoys a simple expression.
\begin{proposition}[Pullback of the $\Lrhos$ metric]
    \label{app:Gtheta_pullback}
    The pullback metric of the flat $\Lrhos$-metric by the map $\theta\mapsto f_\theta$ is the matrix $G_\theta$ defined by
    \begin{equation}
        G_\theta\coloneqq \int_{\bR^d}(D_\theta f_\theta(x))^\top D_\theta f_\theta(x)\dd\rhos(x)\qquad \text{for all }\theta\in\Theta.
    \end{equation}
\end{proposition}
\begin{proof}
    Let $\phi:\theta\to f_\theta$ and $g$ be the flat metric on $\Lrhos$.
    By definition of the pullback metric,
    \begin{multline}
        (\phi^*g)_\theta(\delta\theta,\delta\theta)
        =g_{f_\theta}(d_\theta \phi[\delta\theta],d_\theta \phi[\delta\theta])
        =\int_{\bR^d}\|D_\theta f_\theta(x).\delta\theta\|^2\dd\rhos(x)
        \\
        =(\delta\theta)^\top\Big(\int_{\bR^d}(D_\theta f_\theta(x))^\top D_\theta f_\theta(x)\dd\rhos(x)\Big)\delta\theta
        = (\delta\theta)^\top G_\theta \delta\theta,
    \end{multline}
    which ends the proof.
\end{proof}

We believe the following result to be standard but we could not find any textbook reference for it. We therefore provide a short proof here.
\begin{lemma}[Density of increasing subsets]
    \label{lem:Tn}
    Let $K$ be some subset of some Hilbert space $\cH$, and $(K_n)_n$ be an increasing sequence of nonempty subsets of $\cH$ such that $\smash{\overline {\bigcup_n K_n}=K}$. Then for all $f\in K$, there exists a sequence of elements $\hTn\in K_n$ such that $\hTn\to f$.
\end{lemma}
\begin{proof}
    If there exists some $n\geq0$ such that $f\in K_n$, then it is sufficient to take arbitrary elements $(f_0,f_1,\dots, f_{n-1})$ of $K_0\times K_1\times \dots\times K_{n-1}$ and let $f_k\coloneqq f$ for all $k\geq n$ to get such a desired sequence.
    If $f$ does not belong to any $K_n$, let $d_n\coloneqq \operatorname{dist}(f,K_n)>0$. It is positive and decreasing, hence it converges to some nonnegative real number. This number is necessarily zero, else this would contradict the density assumption.
    Let us now construct the desired sequence, starting at some arbitrary $g_0\in K_0$.
    By density, there exists some element $g_1\in \bigcup_m K_m$ such that $\|f-g_1\|<d_0$. This implies that $g_1$ does not belong to $K_0$, and since the $K_n$ are increasing, $g_1\in K_{p_1}$ for some $p_1\geq 1$. Repeating this process with $d_{p_1}$ yields the existence of some $g_2\in K_{p_2}$ such that $\|f-g_2\|\leq d_{p_1}$, where $p_2\geq p_1+1$. Iterating, one gets the existence of a sequence of $g_k\in g_{p_k}$. To account for the fact that the sequence $p_k$ may not go through all nonnegative integers, we fill in the gaps and create a piecewise constant sequence: for $n\geq0$, define $\hTn\coloneqq g_k$ whenever $p_k\leq n <p_{k+1}$, where $p_0\coloneqq 0$. This ensures that for all $n\geq0$, $\hTn\in K_{p_k}\subset K_n$. Finally, since $p_k\to \infty$ when $k\to\infty$, one has $d_{p_{k}}\to 0$ and therefore $\|f-\hTn\|\to0$ as $n\to\infty$.
\end{proof}

\end{document}